\documentclass[12pt,twoside,reqno]{amsart}
\usepackage{amsmath, amsfonts, amsthm, amssymb,amscd, graphicx, amscd}
\usepackage{esint}
\usepackage{float,epsf}
\usepackage[english]{babel}
\usepackage{enumerate}
\usepackage{tikz}
\usepackage{graphicx} 
\usepackage{srcltx}
\usepackage{geometry}
\usepackage{verbatim}
\usepackage{mathrsfs}
\usepackage{hyperref}
\usepackage{enumitem} 
\newcommand{\R}{\mathbb{R}}

\newtheorem{theorem}{Theorem}[section]
\newtheorem{lemma}[theorem]{Lemma}

\newtheorem{corollary}[theorem]{Corollary}
\newtheorem{remark}[theorem]{Remark}

\newtheorem{question}[theorem]{Question}

\numberwithin{equation}{section}
\numberwithin{figure}{section}

\def\intave#1{\int_{#1}\hbox{\llap{$\raise2.3pt\hbox{\vrule
height.9pt width7pt}\phantom{\scriptstyle{#1}}\mkern-2mu$}}}

\begin{document}
\title{A remark on isolated removable singularity of harmonic maps in dimension two}
\author{Changyou Wang}
\address{Department of Mathematics, Purdue University, West Lafayette, IN 47097, USA}
\email{wang2482@purdue.edu}
\begin{abstract} For a ball $B_R(0)\subset\mathbb{R}^2$, we provide sufficient conditions such that 
a harmonic map $u\in C^\infty(B_R(0)\setminus\{0\}, N)$, with a self-similar bound on its gradient,  belongs to $C^\infty(B_R(0))$. 
Those conditions also guarantee the triviality of such harmonic maps when $R=\infty$.
\end{abstract}
\maketitle
\section{Introduction}

This short note address a question arising from the author's recent study with Bang \cite{BangWang2025}  
on the question of rigidity for the steady (simplified) Ericksen-Leslie system in
$\R^n$, which seeks to answer that if $(u,d)\in C^\infty(\R^n\setminus\{0\}, \R^n\times \mathbb{S}^{n-1})$ solves
\begin{equation}\label{EL1}
\begin{cases}
-\Delta u+u\cdot\nabla u+\nabla p=-\nabla\cdot(\nabla d\odot \nabla d),\\
\nabla\cdot u=0,\\
\Delta d+|\nabla d|^2 d=u\cdot\nabla d,
\end{cases}
\ \ {\rm{in}} \ \ \R^n\setminus\{0\},
\end{equation}
and satisfies a self-similar bound 
\begin{equation}\label{self-similar-bound}
|u(x)|\le \frac{C_1(n)}{|x|}, \ \ |\nabla d(x)|\le \frac{C_2(n)}{|x|},  \ \forall x\in \R^n\setminus\{0\},
\end{equation}
for some constants $C_1(n), C_2(n)>0$, then $(u,\nabla d)\equiv (0,0) \ {\rm{in}}\ \R^n$?

We obtained in \cite{BangWang2025} some partial results towards the question concerning
\eqref{EL1} and \eqref{self-similar-bound} mentioned above. In particular, among other results
we proved that when $n\ge 3$, there exists $\varepsilon_n>0$ such that if $C_1(n), C_2(n)\le\varepsilon_n$ then
$\nabla d\equiv 0$; while $u\equiv 0$ when $n\ge 4$, and is a Landau solution of the steady Navier-Stokes equation
when $n=3$. When $n=2$,  we constructed infinitely many nontrivial solutions of \eqref{EL1} and \eqref{self-similar-bound},
that resemble the so-called Hamel's solutions of steady Navier-Stokes equation in $\R^2$. 

A Liouville theorem on harmonic maps plays an important role in \cite{BangWang2025}, that is, for $n\ge 3$ if
$d\in C^\infty(\R^n\setminus\{0\},N)$ solves the equation of harmonic maps:
\begin{equation}\label{hm0}
\Delta d+A(d)(\nabla d,\nabla d)=0 \ \ {\rm{in}}\ \ \R^n\setminus\{0\}, 
\end{equation}
and  there exists an $\varepsilon_0(n)>0$ such that 
\begin{equation}\label{small-self}
|\nabla d(x)|\le \frac{\varepsilon_0(n)}{|x|}, \ \forall x\in\R^n\setminus\{0\},
\end{equation}
then $d$ must be a constant map. Here $N\subset \R^L$ is a compact smooth Riemann manifold without boundary,
and $A$ denotes the second fundamental form of $N$.

A natural question to ask is that whether this Liouville property remains to be true when $n=2$.  
More precisely, 
\begin{question}  Suppose $d\in C^\infty(\R^2\setminus\{0\},N)$ solves \eqref{hm0} and satisfies \eqref{small-self} for some
small constant $\varepsilon_0(2)$. Then $d$ must be constant.
\end{question}

To the knowledge of the author, this question has not been addressed from the literature.  In contrast with $n\ge 3$, \eqref{small-self} alone
does not guarantee $d$ has locally finite Dirichlet energy in dimension two, 
or, equivalently, $E(d, B_1(0))=\int_{B_1(0)}|\nabla d|^2<\infty$ for the unit ball $B_1(0)\subset\R^2$. Thus, one may apply neither 
the celebrated theorem by Sacks-Uhlenbeck \cite{SacksUhlenbeck1981} on the removability of isolated singularity of harmonic maps in dimension two,
nor the regularity theorem by H\'elein \cite{Helein1991} on weakly harmonic maps in dimension two. Observe that 
$d(x)=\frac{x}{|x|}:\R^2\setminus \{0\}\to \mathbb{S}^1$ is a harmonic map, satisfying $|\nabla d(x)|=\frac{1}{|x|}$ for $x\not=0$
and $E(d, B_1(0))=\infty$,  while $x=0$ is a non-removable singular point. This example indicates that $\varepsilon_0(2)$ in Question 1.1 must be 
sufficiently small. 

\medskip
In this note, we will give a partial answer to Question 1.1. More precisely, let $B_R(0)\subset\R^2$ be the ball in $\R^2$
with center $0$ and radius $R$, we first prove
\begin{theorem} \label{removability} 
There exists an
$\varepsilon_0>0$ such that if
$u:B_R(0)\setminus\{0\}\to N$ is a smooth harmonic map, satisfying
\begin{equation}\label{self-similar1}
|\nabla u(x)|\le \frac{\varepsilon_0}{|x|}, \ \ \forall x\in B_R(0)\setminus\{0\},  
\end{equation}
and if, in addition, there exists $r_i\to 0$ such that
\begin{equation}\label{equal-dist0}
\lim_{i\to\infty} r_i\int_{\partial B_{r_i}(0)} \big(|\frac{\partial u}{\partial r}|^2-\frac{1}{r^2}|\frac{\partial u}{\partial\theta}|^2\big)\,d\sigma=0,
\end{equation}
then $u\in C^\infty(B_R(0),N)$.
\end{theorem}

As a direct consequence of Theorem \eqref{removability}, we establish
\begin{corollary} \label{rigidity}
There exists an $\varepsilon_0>0$ such that if $u\in C^\infty(\mathbb{R}^2\setminus\{0\}, N)$ is a harmonic map, satisfying
\begin{equation}\label{self-similar2}
|\nabla u(x)|\le \frac{\varepsilon_0}{|x|}, \ \ \forall x\in \mathbb{R}^2\setminus\{0\},  
\end{equation}
and if, in addition, there exists $r_i\to 0$ such that
\begin{equation}\label{equal-dist4}
\lim_{i\to\infty} r_i\int_{\partial B_{r_i}(0)} \big(|\frac{\partial u}{\partial r}|^2-\frac{1}{r^2}|\frac{\partial u}{\partial\theta}|^2\big)\,d\sigma=0,
\end{equation}
then $u$ must be a constant map.
\end{corollary}

\section{Proofs}

This section is devoted to the proofs of Theorem \ref{removability} and Corollary \ref{rigidity}. First, we need

\begin{lemma} \label{equal-dist1} If $u\in C^\infty(B_R(0)\setminus\{0\}, N)$ is a harmonic map, then
\begin{equation}\label{equal-dist2}
\phi(r):=r\int_{\partial B_r(0)} \big(|\frac{\partial u}{\partial r}|^2
-\frac{1}{r^2}|\frac{\partial u}{\partial \theta}|^2\big)\,d\sigma
\end{equation}
is constant for $r\in (0, R)$.
\end{lemma} 
\begin{proof} 
Since $u\in C^\infty(B_R(0)\setminus\{0\}, N)$ solves the harmonic map equation \eqref{hm0}, 
for any $0<r_1<r_2<R$, we can multiply (\ref{hm0}) by $x\cdot\nabla u$ and integrate the resulting equation
over $B_{r_2}(0)\setminus B_{r_1}(0)$ to obtain
\begin{align*}
0&=\int_{B_{r_2}(0)\setminus B_{r_1}(0)}\Delta u\cdot(x\cdot\nabla u)=\int_{B_{r_2}(0)\setminus B_{r_1}(0)}
(u_j x_iu_i)_j-|\nabla u|^2
-\frac12 x_j(|\nabla u|^2)_j\\
&=\int_{\partial(B_{r_2}(0)\setminus B_{r_1}(0))}
(x\cdot \nabla u)\cdot( \nu\cdot\nabla u)
-\frac12\int_{\partial(B_{r_2}(0)\setminus B_{r_1}(0))}|\nabla u|^2 x\cdot\nu,
\end{align*}
where $\nu$ denotes the outward unit normal of $\partial(B_{r_2}(0)\setminus B_{r_1}(0))$. This implies that
\begin{align*}
r_2\int_{\partial B_{r_2}(0)} \big(|\frac{\partial u}{\partial r}|^2
-\frac12|\nabla u|^2\big)\,d\sigma
=
r_1\int_{\partial B_{r_1}(0)} \big(|\frac{\partial u}{\partial r}|^2
-\frac12|\nabla u|^2\big)\,d\sigma.
\end{align*}
Since $\displaystyle|\nabla u|^2=|\frac{\partial u}{\partial r}|^2+\frac{1}{r^2}|\frac{\partial u}{\partial \theta}|^2$,
it follows that
\begin{align}
r_2\int_{\partial B_{r_2}(0)} \big(|\frac{\partial u}{\partial r}|^2
-\frac{1}{r^2}|\frac{\partial u}{\partial\theta}|^2\big)\,d\sigma
=
r_1\int_{\partial B_{r_1}(0)} \big(|\frac{\partial u}{\partial r}|^2
-\frac{1}{r^2}|\frac{\partial u}{\partial \theta}|^2\big)\,d\sigma.
\end{align}
This implies \eqref{equal-dist2}. \qed

\begin{remark} It is easy to check that if $\displaystyle d(x)=\frac{x}{|x|}:\R^2\setminus\{0\}\to \mathbb S^1$, then $\phi(r)=-2\pi$ for all $r>0$.
\end{remark}

\medskip
\noindent{\it Proof of Theorem \ref{removability}} : From the condition \eqref{equal-dist0}
and  \eqref{equal-dist2} of Lemma \ref{equal-dist1}, we have that 
\begin{equation}\label{equal-dist3}
    \int_{\partial B_r(0)} |\frac{\partial u}{\partial r}|^2\,d\sigma=\frac{1}{r^2}  \int_{\partial B_r(0)}|\frac{\partial u}{\partial\theta}|^2\,d\sigma
\end{equation}
holds for all $0<r<R$.

We will modify the original argument by Sacks-Uhlenbeck \cite{SacksUhlenbeck1981}
in showing that $x=0$ is a removable singularity for $u$. 

First, we will show that 
$u$ has finite Dirichlet energy, i.e., $u\in H^1(B_R(0))$.
For this, let $0<r_*<R_*\le R$ be two given radius.  Set $K=\big[\frac{\ln(\frac{R_*}{r_*})}{\ln 2}\big]\in \mathbb N$
and define the annuals 
$$\displaystyle A_m=B_{2^mr_*}(0)\setminus 
B_{2^{m-1}r_*}(0),  \  \ 1\le m\le K.$$
Denote the radial harmonic function
$h_m(r):=a_m+b_m\ln r: A_m\to \mathbb{R}^L$, where $a_m$ and $b_m\in\mathbb{R}^L$ are chosen
according to the condition
$$
h_m(2^{m}r_*)=
\fint_{\partial B_{2^{m}r_*}} u\,d\sigma,
\ \
h_m(2^{m-1}r_*)=
\fint_{\partial B_{2^{m-1}r_*}} u\,d\sigma,
$$
where $\displaystyle\fint_{\partial B_r(0)} f\,d\sigma=\frac{1}{2\pi r}\int_{\partial B_r(0)} f\,d\sigma$ 
denotes the average of $f$ over $\partial B_r(0)$. 

Note that the condition (\ref{self-similar1}) implies
$$
{\rm{osc}}_{A_m} u\le C\varepsilon_0, \ \forall 1\le m\le K.
$$
Now, by multiplying (\ref{hm0}) by
$u-h_m$ and integrating the resulting equation over $A_m$, we obtain 
\begin{align*}
&\int_{A_m}|\nabla (u-h_m)|^2
=\int_{\partial A_m} 
\big(\frac{\partial u}{\partial r}-h_m'(r)\big) \cdot(u-h_m)
+\int_{A_m} A(u)(\nabla u,\nabla u)\cdot(u-h_m)\\
&=\int_{\partial B_{2^{m}r_*}(0)} 
\frac{\partial u}{\partial r}\cdot(u-h_m)
-\int_{\partial B_{2^{m-1}r_*}(0)} 
\frac{\partial u}{\partial r}\cdot(u-h_m)\\
&\quad +\int_{A_m} A(u)(\nabla u,\nabla u)\cdot(u-h_m)\\
&\le \int_{\partial B_{2^{m}r_*}(0)}
\frac{\partial u}{\partial r} \cdot(u-h_m)
-\int_{\partial B_{2^{m-1}r_*}(0)}
\frac{\partial u}{\partial r}\cdot(u-h_m)+C\varepsilon_0
\int_{A_m}|\nabla u|^2.
\end{align*}
Since $h_m$ depends only on $r$, 
we can apply \eqref{equal-dist3} to obtain that
$$\int_{A_m}|\nabla (u-h_m)|^2
\ge \int_{A_m}\frac{1}{r^2}
|\frac{\partial u}{\partial \theta}|^2\,d\sigma=\frac12\int_{A_m}
|\nabla u|^2.$$
Hence we have
\begin{equation}\label{energy1}
(\frac12-C\varepsilon_0)
\int_{A_m}
|\nabla u|^2
\le \int_{\partial B_{2^{m}r_*}(0)}
\frac{\partial u}{\partial r} \cdot(u-h_m)
-\int_{\partial B_{2^{m-1}r_*}(0)}
\frac{\partial u}{\partial r}\cdot(u-h_m)
\end{equation}
By summing up \eqref{energy1} over
$1\le m\le K$, we obtain that
\begin{align}\label{energy2}
(\frac12-C\varepsilon_0)
\int_{B_{2^Kr_*}(0)\setminus B_{r_*}(0)}
|\nabla u|^2
&\le 
 \int_{\partial B_{2^Kr_*}(0)}
\frac{\partial u}{\partial r}\cdot(u-h_K)\nonumber\\
&\quad-\int_{\partial B_{r_*}(0)}
\frac{\partial u}{\partial r}\cdot(u-h_1).
\end{align}
By Poincar\`e inequality, \eqref{equal-dist3} and \eqref{self-similar1}, the terms in the right hand side of \eqref{energy2} can 
be estimated by
\begin{align}\label{energy4}
 \big|\int_{\partial B_{2^Kr_*}(0)}
\frac{\partial u}{\partial r}\cdot(u-h_K)\big|
&\le
C
\big(\int_{\partial B_{2^Kr_*}(0)}|\frac{\partial u}{\partial r}|^2\,d\sigma\big)^\frac12
\big(\int_{\partial B_{2^Kr_*}(0)}|u-h_K|^2\,d\sigma\big)^\frac12\nonumber\\
&\le
C 2^Kr_*
\Big(\int_{\partial B_{2^Kr_*}(0)}|\frac{\partial u}{\partial r}|^2\,d\sigma\Big)^\frac12
\Big(\int_{\partial B_{2^Kr_*}(0)}\frac{1}{r^2}|\frac{\partial u}{\partial\theta}|^2\,d\sigma\Big)^\frac12\nonumber\\
&\le C2^Kr_*\int_{\partial B_{2^Kr_*}(0)}|\nabla u|^2\,d\sigma\le C\varepsilon_0^2,
\end{align}
and, similarly,
\begin{align}\label{energy5}
 \big|\int_{\partial B_{r_*}(0)}
\frac{\partial u}{\partial r}\cdot(u-h_1)\big|
\le Cr_*
\int_{\partial B_{r_*}(0)}|\nabla u|^2\,d\sigma\le C\varepsilon_0^2.
\end{align}
Substituting the inequalities \eqref{energy4} and \eqref{energy5}
into (\ref{energy2}) yields 
\begin{equation}\label{energy3}
 (\frac12-C\varepsilon_0)
\int_{B_{2^Kr_*}(0)\setminus B_{r_*}(0)}
|\nabla u|^2\le C\varepsilon_0^2.
\end{equation}
Thus, by choosing $\varepsilon_0<\frac{1}{4C}$ and observing $\frac{R_*}2\le 2^Kr_*\le R_*$, we obtain that 
\begin{equation}\label{energy6}
\int_{B_{\frac{R_*}2}(0)\setminus B_{r_*}(0)}
|\nabla u|^2\le C\varepsilon_0^2.
\end{equation}
Since \eqref{energy6} holds for any two $0<r_*<R_*\le R$, we conclude that 
\begin{equation}\label{energy7}
\int_{B_{\frac{R}2}(0)}|\nabla u|^2\le C\varepsilon_0^2<\infty. 
\end{equation}

Next, with the help of \eqref{energy7}, we can repeat the above arguments to obtain the H\"older
continuity of $u$ near $x=0$. In fact,  after labeling $r=2^K r_*$ so that $r_*=2^{-K} r$,
\eqref{energy2}, \eqref{energy4} and \eqref{energy5} imply that for any $0<r<R$,
\begin{equation}\label{energy8}
\int_{B_r(0)\setminus B_{2^{-K}r}(0)} 
|\nabla u|^2
\le Cr\int_{\partial B_r(0)}|\nabla u|^2\,d\sigma+C2^{-K}r\int_{\partial B_{2^{-K}r}(0)}|\nabla u|^2\,d\sigma.    
\end{equation}
On the other hand, it follows from \eqref{energy7} that
$$
\lim_{K\to\infty} 2^{-K}r\int_{\partial B_{2^{-K}r}(0)}|\nabla u|^2\,d\sigma=0.
$$
Hence, after sending $K\to\infty$ in \eqref{energy8}, we obtain that for any $0<r<R$, 
\begin{equation}\label{energy9}
\int_{B_r(0)} |\nabla u|^2\le Cr\int_{\partial B_r(0)}|\nabla u|^2\,d\sigma.
\end{equation}
This implies that there exists an
$\alpha\in (0,1)$ such that 
$$
\int_{B_r(0)}|\nabla u|^2\le (\frac{r}{R})^{2\alpha}\int_{B_{\frac{R}2}(0)}|\nabla u|^2, \ \ 0<r\le \frac{R}2.
$$
Hence $u\in C^\alpha(B_\frac{R}2(0))$. By the higher order regularity of harmonic maps,
$u\in C^\infty(B_{\frac{R}2}(0))$ (see, for example, \cite{SacksUhlenbeck1981}).    
\end{proof}

\bigskip
\noindent{\it Proof of Corollary \ref{rigidity}} :  It follows from Theorem \ref{removability} and \eqref{energy9} that 
$u\in C^\infty(\mathbb{R}^2)$, and satisfies
\begin{equation}\label{energy10}
\int_{B_R(0)}
|\nabla u|^2
\le CR\int_{\partial B_R(0)}|\nabla u|^2\,d\sigma\le C\varepsilon_0^2,    \ \forall R>0.
\end{equation}
By choosing sufficiently small $\varepsilon_0$ in (\ref{energy10}) and applying the $\varepsilon_0$-gradient estimate for harmonic maps, we obtain that
$$\|\nabla u\|_{L^\infty(B_R(0))}\le \frac{C\varepsilon_0}{R}, \ \forall R>0.$$
This, after sending $R\to\infty$, yields $u$ must be constant.    
\qed

\bigskip
\bigskip
\noindent{\bf Acknowledgements}. The author is partially supported by NSF grant 2101224 and a Simons Travel Grant.


\bigskip
\begin{thebibliography}{99}
\bibitem{BangWang2025} J. Bang, C. Y. Wang, {\em On rigidity of the steady Ericksen-Leslie system}. arXiv:2502.05326.

\bibitem{Helein1991}  F. H\'elein; {\it Regularite des appliations faiblement harmoniques entre unne
surface et une variete riemannienne}, C. R. Aad. Si. Paris {\bf 312} (1991), 591-596.

\bibitem{SacksUhlenbeck1981} J. Sacks, K. Uhlenbeck, {\it The existence of minimal immersions of $2$-spheres}.
Annals Math.,  {\bf 113} (1981), 1-24.

\end{thebibliography}
\end{document}